\documentclass[american,english]{article}
\usepackage[T1]{fontenc}
\usepackage[latin9]{inputenc}
\usepackage{units}
\usepackage{amsmath}
\usepackage{amsthm}
\usepackage{amssymb}
\usepackage{setspace}
\makeatletter
\theoremstyle{plain}
\newtheorem{thm}{\protect\theoremname}[section]
\theoremstyle{remark}
\newtheorem*{acknowledgement*}{\protect\acknowledgementname}
\theoremstyle{definition}
\newtheorem{defn}[thm]{\protect\definitionname}
\theoremstyle{plain}
\newtheorem{prop}[thm]{\protect\propositionname}
\theoremstyle{plain}
\newtheorem{lem}[thm]{\protect\lemmaname}
\theoremstyle{plain}
\newtheorem{cor}[thm]{\protect\corollaryname}
\newcommand{\lyxaddress}[1]{
	\par {\raggedright #1
	\vspace{1.4em}
	\noindent\par}
}

\usepackage{enumitem}
\setenumerate{label=(\roman{enumi})}
\date{}

\makeatother

\usepackage{babel}
\addto\captionsamerican{\renewcommand{\acknowledgementname}{Acknowledgement}}
\addto\captionsamerican{\renewcommand{\corollaryname}{Corollary}}
\addto\captionsamerican{\renewcommand{\definitionname}{Definition}}
\addto\captionsamerican{\renewcommand{\lemmaname}{Lemma}}
\addto\captionsamerican{\renewcommand{\propositionname}{Proposition}}
\addto\captionsamerican{\renewcommand{\theoremname}{Theorem}}
\addto\captionsenglish{\renewcommand{\acknowledgementname}{Acknowledgement}}
\addto\captionsenglish{\renewcommand{\corollaryname}{Corollary}}
\addto\captionsenglish{\renewcommand{\definitionname}{Definition}}
\addto\captionsenglish{\renewcommand{\lemmaname}{Lemma}}
\addto\captionsenglish{\renewcommand{\propositionname}{Proposition}}
\addto\captionsenglish{\renewcommand{\theoremname}{Theorem}}
\providecommand{\acknowledgementname}{Acknowledgement}
\providecommand{\corollaryname}{Corollary}
\providecommand{\definitionname}{Definition}
\providecommand{\lemmaname}{Lemma}
\providecommand{\propositionname}{Proposition}
\providecommand{\theoremname}{Theorem}

\begin{document}
\title{Moments of finitary factor maps between Bernoulli processes}
\author{Uri Gabor\thanks{Supported by ISF grants 1702/17 and 3056/21}}
\maketitle
\begin{abstract}
The problem of what moments can exist for the coding radius of a finitary
map between two i.i.d. processes, has been extensively studied in
the case of $\mathbb{Z}$-processes. Here we treat this problem for
factor maps between $\mathbb{Z}^{d}$-processes ($d>1$). By modeling
the homomorphism with a map between spaces of finite sequences, we
extend Harvey and Peres' result, showing that for a finitary homomorphism
between two i.i.d. processes of equal entropy, if the coding radius
of the map has a finite $\frac{d}{2}$-moment, then the two processes
share the same informational variance. We use our modeling technique
to prove a ``Schmidt-type theorem'' - that in case the above homomorphism
has a coding radius of exponential tails, then the two processes are
essentially the same. This result appears to be new even for the one-dimensional
case, addressing a question of Angel and Spinka \cite{key-8}.
\end{abstract}

\section{Introduction}

The celebrated result of Keane and Smorodinsky, that two i.i.d. processes
of the same entropy are finitary isomorphic, was immediately followed
by the question of whether there is in general such a finitary isomorphism
which in addition exhibits a finite expected coding radius. Here,
the coding radius of a map $\phi:(A^{\mathbb{Z}^{d}},\mu)\rightarrow(B^{\mathbb{Z}^{d}},\nu)$
is the random variable 
\[
R_{\phi}(x)=\min\{n\in\mathbb{N}\cup\{\infty\}:\text{For \ensuremath{\mu}-a.e. }a,(a_{i})_{i\in B(0,n)}=(x_{i})_{i\in B(0,n)}\Rightarrow\phi(a)_{0}=\phi(x)_{0}\}.
\]
In case $R_{\phi}$ is a.s. finite, $\phi$ is said to be finitary,
while if the coding radius has finite expectation $\mathbf{E}[R_{\phi}]<\infty$
or finite $\alpha$-moment $\mathbf{E}[R_{\phi}^{\alpha}]<\infty$
for some $\alpha>0$, then $\phi$ is said to exhibit a finite expected
coding radius or finite $\alpha$-moment respectively.

This question was first answered by Parry \cite{key-4}, who showed
that for an isomorphism $\phi$ between two $\mathbb{Z}$-indexed
i.i.d. processes $X$ and $Y$, if the coding radii of both $\phi$
and $\phi^{-1}$ have finite expectation, then the one-dimensional
marginals of $X$ and $Y$ share the same informational variance.
That is, writing $p=\left(p_{a}\right)_{a\in A}$ and $q=\left(q_{b}\right)_{b\in B}$
for the marginal probabilities of $X_{0}$ and $Y_{0}$ respectively,
the information functions of $X_{0}$ and $Y_{0}$ are the random
variables $I_{p},I_{q}$ with domains in $A$ and $B$ respectively,
defined by
\[
\begin{array}{ccc}
I_{p}(a)=\log p_{a}^{-1} & \hfill & I_{q}(b)=\log q_{b}^{-1}\end{array},
\]
 and the informational variance of $X_{0}$ is just $\text{Var}(I_{X_{0}})$.
Assuming the processes have equal entropy, having the same informational
variance reduces to the following equality:
\begin{equation}
\sum_{a\in A}p_{a}\log^{2}p_{a}^{-1}=\sum_{b\in B}q_{b}\log^{2}q_{b}^{-1}\label{eq:-6}
\end{equation}
which does not hold for two i.i.d. processes of equal entropy in general.
Krieger \cite{key-3} gave examples of a different nature for circumstances
where the isomorphism between two i.i.d. processes, although it exists,
cannot have a finite expected moment. Based on these works together
with \cite{key-5}, Schmidt \cite{key-6} showed that two i.i.d. processes
can be finitarily isomorphic with finite expected coding radius only
if their marginal probabilities are permutations of each other.

The question for other moments was addressed by Harvey and Peres in
\cite{key-1}, where it is shown that having merely a homomorphism
$\phi$ between two i.i.d. processes of equal entropy that exhibits
a finite half moment, i.e. 
\[
\mathbf{E}[R_{\phi}^{\frac{1}{2}}]<\infty,
\]
already restricts the marginal probabilities of the two processes
to share the same informational variance as in (\ref{eq:-6}). Recently
we showed \cite{key-7} that this restriction is tight, in the sense
that any two i.i.d. processes of equal entropy are isomorphic by a
finitary map $\phi$ s.t. both $R_{\phi}$ and $R_{\phi^{-1}}$ exhibit
finite $\theta$-moments for all $\theta<\frac{1}{2}$.

In this work we treat this problem for $\mathbb{Z}^{d}$-indexed i.i.d.
processes. Most of the above methods heavily rely on the special geometric
structure of $\mathbb{Z}$, in which the boundary of an interval $I\subset\mathbb{Z}$
is of a constant size $|\partial I|=O(1)$, regardless of the size
of the interval $|I|$. In contrast, for a cube $B(\rho,n):=[\rho-n,\rho+n]^{d}$
in $\mathbb{Z}^{d}$, the boundary is of size $|\partial B(\rho,n)|=\Theta(n^{d-1})$.
Thus, these results cannot be extended in a straight forward manner
to higher dimensional processes. (However, it should be pointed out
that the mentioned results of Krieger can be easily extended to higher
dimensions, providing the first examples for $\mathbb{Z}^{d}$-indexed
i.i.d. processes of the same entropy with no isomorphism between them
exhibiting a finite expected coding radius.) Roughly speaking, the
basic idea in all the above results is that the bound on the coding
radius implies that for a given interval $I\subset\mathbb{Z}$, with
high probability, a large fraction of the image symbols $\left(\phi(X)_{i}\right)_{i\in I}$
can be encoded by reading only $\left(X_{i}\right)_{i\in I}$. This
imposes a strong connection between the distributions of $\left(X_{i}\right)_{i\in I}$
and $\left(Y_{i}\right)_{i\in I}$, which in turn restricts the marginal
distributions $X_{0}$ and $Y_{0}$ to share a common behavior. Thus
for example, the probability that a given site $\phi(X)_{i}$ is not
determined by $\left(X_{i}\right)_{i\in I}$ is at most $\mathbf{P}[R_{\phi}>d(i,I^{c})]$,
and the expected number of indices $i\in I$ where $\phi(X)_{i}$
is not determined by $\left(X_{i}\right)_{i\in I}$ is bounded by
\[
\sum_{i\in I}\mathbf{P}[R_{\phi}>d(i,I^{c})]=O(\sum_{n=1}^{\infty}\mathbf{P}[R_{\phi}>n])=O(\mathbf{E}[R_{\phi}])
\]
so that having a finite expected coding radius implies a good bound
on the expected number of non-encoded sites. In contrast, applying
the same method for a cube in $\mathbb{Z}^{d}$ with $d>1$ ends up
with a useless bound that tends to infinity as the size of the cube
grows.

To overcome this problem, we abandon the intrinsic structure of the
mapping $\phi:X\longmapsto Y$, and use a finite approximation of
it, $\hat{\phi}:\hat{X}\longmapsto\hat{Y}$, where $\hat{X}$ and
$\hat{Y}$ are processes labeled over a $d$-dimensional torus $\mathbb{T}_{N}^{d}=\mathbb{Z}^{d}/(N\mathbb{Z})^{d}$.
In these processes, the set of indices $\mathbb{T}_{N}^{d}$ is finite,
yet with no boundary points. The new processes has enough in common
with the original processes, so that imposing restrictions on the
connection between the finite processes can be then pulled back to
give a connection between the original processes. This method is explained
in Section \ref{sec:The--model}. With this in hand, we are able to
extend the result of Harvey and Peres mentioned before to all dimensions:
\begin{thm}
\label{thm:For-any-i.i.d.-1}For any i.i.d. $\mathbb{Z}^{d}$-processes
$X$ and $Y$ of equal entropy $h$, if there exists a finitary map
$\phi:X\rightarrow Y$ with $\mathbf{E}[R_{\phi}^{d/2}]<\infty$,
then $\text{Var}(I(X_{0}))=\text{Var}(I(Y_{0}))$.
\end{thm}

It seems that the analog in high dimension for having a finite expected
moment is having a finite $d$-moment. In this case, we could not
figure out whether Schmidt's theorem extends to higer dimensions or
not. However, in Section \ref{sec:A-Schmidt-type-theorem} we use
our modeling technique to prove a ``Schmidt-type'' theorem for the
existence of finitary factor map with exponential tails between two
i.i.d. processes of the same entropy. Harvey, Holroyd, Peres, and
Romik proved the existence of such a map between any two i.i.d. processes,
provided that the factor process has entropy strictly lesser than
the domain process's entropy (they prove their result to any Markov
chain with finite state space). Since then, more families of processes
were shown to be finitary factors with exponential tails of an i.i.d.
process. Common to all these results is the assumption on the entropy
difference between the two processes in question. Angel and Spinka
\cite{key-8} raised the question of whether this assumption can be
removed for any Markov Chain. Our result, restricted to $d=1$, addresses
this question for the case of i.i.d. processes, showing that the removal
of the entropy-gap assumption is possible only when the two processes
are completely identical:
\begin{thm}
\label{thm:Let--and-1}Let $X$ and $Y$ be two i.i.d. $\mathbb{Z}^{d}$-processes
of equal entropy $h$, with marginal measures $p,q$ respectively.
Suppose that for some constant $c>0$ there is a homomorphism $\phi(X)=Y$
that satisfies $\mathbf{P}[R_{\phi}>n]=O(e^{-cn})$. Then $p=q$ up
to permutation on the symbols.
\end{thm}

\begin{acknowledgement*}
This paper is part of the author\textquoteright s Ph.D. thesis, conducted
under the guidance of Professor Michael Hochman, whom I would like
to thank for all his support and advice.
\end{acknowledgement*}

\section{\protect\label{sec:The--model}The $\mathbb{T}_{N}^{d}$-model}

For a positive integer $N$, write $\mathbb{T}_{N}^{d}:=\mathbb{Z}^{d}/(N\mathbb{Z})^{d}$
for the ``integral'' $d$-dimensional torus of size $N^{d}$. Let
$X\sim(A^{\mathbb{Z}^{d}},\mu)$ and $Y\sim(B^{\mathbb{Z}^{d}},\nu)$
be two $\mathbb{Z}^{d}$-indexed i.i.d. processes with marginal measures
$p$ and $q$ respectively, and suppose there is a finitary factor
map $\phi:X\longmapsto Y$. We introduce a $\mathbb{T}_{N}^{d}$-model
for this homomorphism, which consists of a $\mathbb{T}_{N}^{d}$-indexed
i.i.d. process $\hat{X}$ and a factor map $\hat{\phi}^{n}:\hat{X}\longmapsto\hat{\phi}^{n}(\hat{X})$,
approximating the behavior of $X$ and its factor map $\phi$. Denoting
by $R_{\phi}$ the coding radius of $\phi$, for any $n\in\mathbb{N}$
and $\mu$-a.e. $x\in A^{\mathbb{Z}^{d}}$, we define 
\[
\phi_{0}^{n}(x)=\begin{cases}
\left(\phi(x)\right)_{0} & R_{\phi}(x)\leq n\\*
* & \text{Otherwise}
\end{cases}
\]
and extend it to a map $\phi^{n}:A^{\mathbb{Z}^{d}}\rightarrow(B\cup\{*\})^{\mathbb{Z}^{d}}$
by letting $\left(\phi^{n}(x)\right)_{u}=\phi_{0}^{n}(T_{-u}x)$ for
any $u\in\mathbb{Z}^{d}$ (where $T_{-u}$ denotes translation by
$-u$, defined by $T_{-u}x(v)=x(v+u)$). Clearly, the coding radius
of $\phi_{0}^{n}$ is bounded by $n$, thus one can define in the
same manner the map $\hat{\phi}^{n}:A^{\mathbb{T}_{N}^{d}}\rightarrow(B\cup\{*\})^{\mathbb{T}_{N}^{d}}$
for any $N\geq2n+1$: Given $\hat{x}\in A^{\mathbb{T}_{N}^{d}}$ and
$u\in\mathbb{Z}^{d}$, abbreviate $\hat{x}_{u}$ for $\hat{x}_{u\,\mod{\,}N}$,
and denote by $[\hat{x}]\subseteq A^{\mathbb{Z}^{d}}$ the cylinder
set
\[
[\hat{x}]=\left\{ x\in A^{\mathbb{Z}^{d}}:(x_{u}=\hat{x}_{u})_{u\in[1,N]^{d}}\right\} .
\]

For any $\hat{x}\in A^{\mathbb{T}_{N}^{d}}$ there exists a unique
value in $B\cup\{*\}$ which $\phi^{n}(x)_{0}$ is equal to for a.e.
$x\in[\hat{x}]$. Define $\left(\hat{\phi}^{n}(\hat{x})\right)_{0}$
to be that value, and for any $u\in\mathbb{T}_{N}^{d}$, let $\left(\hat{\phi}^{n}(\hat{x})\right)_{u}=\left(\phi^{n}(T_{-u}\hat{x})\right)_{0}$
(here $T_{-u}$ is the translation by $u$ operator on $A^{\mathbb{T}_{N}^{d}}$,
defined by $T_{-u}x(v)=x(v+u)$). 

We now can define the $(\mathbb{T}_{N}^{d},n)$-modeling of $\phi$:
\begin{defn}
\label{def:Letting--and}Letting $\hat{X}\sim(A^{\mathbb{T}_{N}^{d}},p^{\mathbb{T}_{N}^{d}})$
and $\hat{Y}=\hat{\phi}^{n}(\hat{X})$,\textit{ the $(\mathbb{T}_{N}^{d},n)$-modeling
of $\phi$} is the homomorphism
\[
\hat{\phi}^{n}:\hat{X}\longmapsto\hat{Y}.
\]
\end{defn}

Denote by $\mathcal{J}=\mathcal{J}_{N,n}\subset\mathbb{T}_{N}^{d}$
the random set of all indices $j$ for which $\hat{\phi}^{n}(\hat{X})_{j}\in B$,
and by $\mathcal{J}^{c}$ the set of indices $j$ for which $\hat{\phi}^{n}(\hat{X})_{j}=*$.
The following proposition summarizes several conclusions whose proofs
are immediate from the above definitions:
\begin{prop}
\label{prop:}
\begin{enumerate}
\item \label{enu:For-any-,}For any $u\in\mathbb{T}_{N}^{d}$, 
\[
\mathbf{P}[\hat{\phi}^{n}(\hat{X})_{u}=*]=\mathbf{P}[R_{\phi}>n]
\]
and
\[
\mathbf{E}[|\mathcal{J}_{N,n}^{c}|]=N^{d}\mathbf{P}[R_{\phi}>n].
\]
In particular, if $\mathbf{P}[R_{\phi}>n]=o(n^{-d\alpha})$ for some
$\alpha>0$, one can find a sequence $\delta_{N}\searrow0$ so that
\begin{equation}
\mathbf{E}[|J_{N,\lfloor\delta_{N}N\rfloor}^{c}|]=o(N^{d(1-\alpha)}).\label{eq:-2}
\end{equation}
\item \label{enu:Suppose--is}Given $N,n$, suppose $\mathcal{S}\subset\mathbb{Z}^{d}$
is a set of elements that satisfies 
\[
B(\mathcal{S},n):=\bigcup_{s\in\mathcal{S}}[s-n,s+n]^{d}\subseteq u_{0}+[1,N]^{d}
\]
 for some $u_{0}\in\mathbb{Z}^{d}$. Then there exists a coupling
$\lambda_{\mathcal{S}}$ of $\left(\hat{\phi}^{n}(\hat{X})_{u}\right)_{u\in\mathcal{S}}$
with $\left(Y_{u}\right)_{u\in\mathcal{S}}\sim q^{\mathcal{S}}$ s.t.
for any $u\in\mathcal{S}$,
\[
\{\hat{\phi}^{n}(\hat{X})_{u}\neq Y_{u}\}=\{\hat{\phi}^{n}(\hat{X})_{u}=*\}.
\]
\end{enumerate}
\end{prop}

The next lemma points out an important connection between the original
factor $Y=\phi(X)$, whose law is denoted by $\nu$, and the new factor
$\hat{\phi}^{n}(\hat{X})$, whose law will be denoted by $\hat{\nu}:=\hat{\phi}_{*}^{n}p^{\mathbb{T}_{N}^{d}}$.
For any $\hat{y}\in(B\cup\{*\})^{\mathbb{T}_{N}^{d}}$, write
\[
[\hat{y}]=\left\{ y\in B^{\mathbb{Z}^{d}}:\forall u\in[1,N]^{d},\;\hat{y}_{u}\in\{*,y_{u}\}\right\} .
\]
For any $\hat{x}\in A^{\mathbb{T}_{N}^{d}}$, it is clear that $\mu([\hat{x}])$
is bounded by the $\hat{\nu}$-measure of its image under $\hat{\phi}^{n}$:
\[
\mu([\hat{x}])=p^{\mathbb{T}_{N}^{d}}(\hat{x})\leq p^{\mathbb{T}_{N}^{d}}\left(\left(\hat{\phi}^{n}\right)^{-1}(\hat{\phi}^{n}(\hat{x}))\right)=\hat{\nu}(\hat{\phi}^{n}(\hat{x})).
\]
Curiously, the quantity $\mu([\hat{x}])$ is actually bounded by the
$\nu$-measure of $[\hat{\phi}^{n}(\hat{x})]$, the set of all points
$y\in B^{\mathbb{Z}^{d}}$ that agree with $\hat{\phi}^{n}(\hat{x})$
on $u\in[1,N]^{d}$ whenever $\hat{\phi}^{n}(\hat{x})_{u}\in B$.
\begin{lem}
\label{lem:For-any-,}For any $\hat{x}\in A^{\mathbb{T}_{N}^{d}}$,
\[
\mu([\hat{x}])\leq\nu([\hat{\phi}^{n}(\hat{x})]).
\]
\end{lem}

\begin{proof}
Fix $m\in\mathbb{N}$ and define $w\in A^{[1,mN]^{d}}$ and $v\in(B\cup\{*\})^{[N+1,(m-1)N]^{d}}$
by 
\begin{align*}
w_{u} & =\hat{x}_{u}\hfill(\forall u\in[1,mN]^{d})\\
v_{u} & =\hat{\phi}^{n}(\hat{x})_{u}\qquad(\forall u\in[N+1,(m-1)N]^{d})
\end{align*}

We further denote by $[w]\subset A^{\mathbb{Z}^{d}}$ the cylinder
set of those points that agree with $w$ on all indices $u\in[1,mN]^{d}$.
Similarly, denote by $[v]\subset B^{\mathbb{Z}^{d}}$ the cylinder
set of those points that agree with $v$ on all indices $u\in[N+1,(m-1)N]^{d}$
that satisfy $v_{u}\in B$. One has 
\begin{align}
\mu([w]) & =\left(\mu([\hat{x}])\right)^{m^{d}}\label{eq:}\\
\nu([v]) & =\left(\nu([\hat{\phi}^{n}(\hat{x})])\right)^{(m-2)^{d}}\label{eq:-1}
\end{align}
and for any $x\in[w]$, its image satisfies $\phi(x)\in[v]$. As $\phi$
is measure preserving, this implies that
\[
\mu([w])\leq\mu(\phi^{-1}[v])=\nu([v])
\]
Using the above identities (\ref{eq:}), (\ref{eq:-1}) and taking
the $m^{d}$th root of both sides gives
\[
\left(\mu([\hat{x}])\right)\leq\left(\nu([\hat{\phi}^{n}(\hat{x})])\right)^{\left(\frac{m-2}{m}\right)^{d}}
\]
and taking $m\rightarrow\infty$ implies the claim.
\end{proof}
The next result is standard. It translates the finite moments property
of a random variable into a tail bound, which has a more geometric
interpretation.
\begin{prop}
\label{prop:Let--be}Let $\phi:X\rightarrow Y$ be a finitary map
between some processes $X,Y$, and suppose that $\mathbf{E}[R^{\alpha}]<\infty$
for some $\alpha>0$. We have then
\[
\mathbf{P}[R_{\phi}>n]=o(n^{-\alpha}).
\]
\end{prop}

\begin{proof}
The $\alpha$-moment of $R_{\phi}$ can be written as
\begin{align*}
\mathbf{E}[R_{\phi}^{\alpha}] & =\int_{0}^{\infty}\mathbf{P}[R_{\phi}^{\alpha}(x)\geq t]dt\\
 & =\alpha\int_{0}^{\infty}s^{\alpha-1}\mathbf{P}[R_{\phi}\geq s]ds\\
 & =\Omega\left(\sum_{n=1}^{\infty}n^{\alpha-1}\mathbf{P}[R_{\phi}>n]\right)
\end{align*}
thus by assumption, the last series converges and finite. Consequently,
along a set $\mathcal{N}\subset\mathbb{N}$ of density 1, the sequence
$a_{n}=n^{\alpha-1}\mathbf{P}[R_{\phi}>n]$ converges to zero faster
than $\frac{1}{n}$, i.e.
\[
\mathbf{P}[R_{\phi}>n]=a_{n}n^{1-\alpha}\underset{n\in\mathcal{N}}{=}o(n^{-\alpha}).
\]
Since $\mathcal{N}$ is of density 1, for any large enough $n$ there
is some $k\in\mathcal{N}\cap[\frac{n}{2},n]$, thus
\[
\mathbf{P}[R_{\phi}>n]\leq\mathbf{P}[R_{\phi}>k]=o((n/2)^{-\alpha})=o(n^{-\alpha})
\]
which yields the claim.
\end{proof}
The last proposition together with Proposition \ref{prop:} ensures
the following corollary:
\begin{cor}
\label{cor:Suppose-that-.}Suppose that $\mathbf{E}[R_{\phi}^{d/2}]<\infty$.
There is a sequence $\delta_{N}\searrow0$ so that 
\[
\mathbf{E}[|J_{N,\lfloor\delta_{N}N\rfloor}^{c}|]=o(N^{d/2}).
\]
\end{cor}

\section{Finite $\frac{d}{2}$-moment}

In this section we will prove Theorem \ref{thm:For-any-i.i.d.-1}.

We denote by $h$ the entropy of $X$, the i.i.d. process in the domain
of $\phi$ in our consideration. That is,
\[
h=h(p)=-\sum_{i=1}^{r}p_{i}\log(p_{i})
\]
where $p=(p_{1},..,p_{r})$ is the marginal measure of $X$. 

In what follows, we will use the notation $c^{+}:=c\vee0$ for any
real $c$. 
\begin{prop}
\label{prop:The-following-holds:}Suppose that $\mathbf{E}[R_{\phi}^{d/2}]<\infty$
and let $\delta_{N}\searrow0$ be as in corollary \ref{cor:Suppose-that-.}.
For each large enough $N\in\mathbb{N},$ consider the $(\mathbb{T}_{N}^{d},\lfloor\delta_{N}N\rfloor)$-modeling
of $\phi$. The following holds:
\begin{enumerate}
\item \label{enu:}$\frac{\mathbf{E}[|\mathcal{J}_{N,\lfloor\delta_{N}N\rfloor}^{c}|]}{\sqrt{|\mathbb{T}_{N}^{d}|}}=o(1).$
\item \label{enu:-1}$\mathbf{E}\left[\frac{1}{\sqrt{|\mathbb{T}_{N}^{d}|}}{\displaystyle \sum_{u\in\mathbb{T}_{N}^{d}}}(h+\log p(\hat{X}_{u}))\right]^{+}=\frac{\text{Var}(\log p(X_{0}))}{\sqrt{2\pi}}\pm o(1)$
\item \label{enu:-2}${\displaystyle \mathbf{E}\left[\frac{1}{\sqrt{|\mathbb{T}_{N}^{d}|}}\sum_{\left\{ u\in\mathbb{T}_{N}^{d}:\hat{Y}_{u}\in B\right\} }(h+\log q(\hat{Y}_{u}))\right]^{+}}=\frac{\text{Var}(\log q(Y_{0}))}{\sqrt{2\pi}}\pm o(1)$
\end{enumerate}
\end{prop}

\begin{proof}
The claim in \ref{enu:} follows from Corollary \ref{cor:Suppose-that-.}.

By the central limit theorem, we have the following weak convergence
\[
\frac{1}{\sqrt{|\mathbb{T}_{N}^{d}|}}\sum_{u\in\mathbb{T}_{N}^{d}}(h+\log p(\hat{X}_{u}))\underset{{\scriptscriptstyle N\rightarrow\infty}}{\overset{d}{\longrightarrow}}Z
\]
where Z is a normal random variable with zero mean and variance equal
to $\text{Var}(\log p(X_{0}))$. Moreover, as $x^{+}=o_{x\rightarrow\infty}(x^{2})$
and for all $N$ 
\[
\mathbf{E}\left[\frac{1}{\sqrt{|\mathbb{T}_{N}^{d}|}}\sum_{u\in\mathbb{T}_{N}^{d}}(h+\log p(\hat{X}_{u}))\right]^{2}=\text{Var}(\log p(X_{0}))<\infty,
\]
it follows that
\[
\mathbf{E}\left[\frac{1}{\sqrt{|\mathbb{T}_{N}^{d}|}}\sum_{u\in\mathbb{T}_{N}^{d}}(h+\log p(\hat{X}_{u}))\right]^{+}\overset{{\scriptscriptstyle N\rightarrow\infty}}{\longrightarrow}\mathbf{E}\left[Z\right]^{+}=\frac{\text{Var}(\log p(X_{0}))}{\sqrt{2\pi}}
\]
(see e.g. \cite{key-9} Exercise 3.2.5.), and \ref{enu:-1} follows.

For the respective claim in \ref{enu:-2} for the sequence $\left(\log q(\hat{Y}_{u})\right)_{u\in\mathbb{T}_{N}^{d}}$,
we split $[1,N]^{d}$ into $2^{d}+1$ sets as follows: write $\mathcal{S}_{0}=[\delta_{N}N,N(1-\delta_{N})]^{d}$,
and enumerating $\{0,N/2\}^{d}=\{v_{1},...,v_{2^{d}}\}$, let 
\[
\mathcal{S}_{i}=(v_{i}+[1,N/2]^{d})\backslash\mathcal{S}_{0}.
\]

We have that 
\[
B(\mathcal{S}_{0},\delta_{N}N)\subset[1,N]^{d},
\]
and for all $1\leq i\leq2^{d}$ one has that 
\[
B(\mathcal{S}_{i},\delta_{N}N)\subset v_{i}-\delta_{N}\overrightarrow{N}+[1,N]^{d}.
\]
(as we took $N$ to be large enough, we can assume that $\delta_{N}<\frac{1}{2}$
so that the last claim clearly takes place). Applying the coupling
as in \ref{enu:Suppose--is} of Proposition \ref{prop:}, for each
$0\leq i\leq2^{d}$ we have
\begin{equation}
\mathbf{E}\left[\sum_{{\scriptscriptstyle u\in\mathcal{S}_{i};\hat{Y}_{u}\in B}}(h+\log q(\hat{Y}_{u}))\right]^{+}=\mathbf{E}\left[\sum_{u\in\mathcal{S}_{i}}(h+\log q(Y_{u}))\right]^{+}+O(\mathbf{E}|\mathcal{J}_{N,\lfloor\delta_{N}N\rfloor}^{c}\cap\mathcal{S}_{i}|).\label{eq:-7}
\end{equation}

Scaling the first summand by $\frac{1}{\sqrt{|\mathcal{S}_{i}|}}$,
its limit when $N\rightarrow\infty$ is again given by the central
limit theorem:
\begin{align*}
\mathbf{E}\left[\frac{1}{\sqrt{|\mathcal{S}_{i}|}}\sum_{u\in\mathcal{S}_{i}}(h+\log q(Y_{u}))\right]^{+}\longrightarrow\frac{\text{Var}(\log q(Y_{0}))}{\sqrt{2\pi}}.
\end{align*}

Thus, scaling the first summand on the right of (\ref{eq:-7}) by
$\frac{1}{\sqrt{|\mathbb{T}_{N}^{d}|}}$, the expression differs between
$\mathcal{S}_{0}$ and the other sets $\left\{ \mathcal{S}_{i}\right\} _{1\leq i\leq2^{d}}$:
For $\mathcal{S}_{0}$ we have 
\[
\sqrt{\nicefrac{|\mathcal{S}_{0}|}{|\mathbb{T}_{N}^{d}|}}=1-O(\delta_{N})=1-o(1),
\]
while for all $1\leq i\leq2^{d}$ we have 
\[
\sqrt{\nicefrac{|\mathcal{S}_{i}|}{|\mathbb{T}_{N}^{d}|}}=O(\sqrt{\delta_{N}})=o(1),
\]
and so
\begin{align*}
\frac{1}{\sqrt{|\mathbb{T}_{N}^{d}|}}\mathbf{E}\left[\sum_{u\in\mathcal{S}_{i}}(h+\log q(Y_{u}))\right]^{+} & =\sqrt{\frac{|\mathcal{S}_{i}|}{|\mathbb{T}_{N}^{d}|}}\mathbf{E}\left[\frac{1}{\sqrt{|\mathcal{S}_{i}|}}\sum_{u\in\mathcal{S}_{i}}(h+\log q(Y_{u}))\right]^{+}\\
 & =\begin{cases}
\frac{\text{Var}(\log q(Y_{0}))}{\sqrt{2\pi}}+o(1) & i=0\\
o(1) & 1\leq i\leq2^{d}
\end{cases}
\end{align*}
and overall we get
\[
\begin{array}{r}
{\displaystyle \mathbf{E}\left[\frac{1}{\sqrt{|\mathbb{T}_{N}^{d}|}}\sum_{{\scriptscriptstyle u\in\mathbb{T}_{N}^{d};\hat{Y}_{u}\in B}}(h+\log q(\hat{Y}_{u}))\right]^{+}=\frac{1}{\sqrt{|\mathbb{T}_{N}^{d}|}}\mathbf{E}\left[\sum_{i=0}^{2^{d}}\sum_{{\scriptscriptstyle u\in\mathcal{S}_{i};\hat{Y}_{u}\in B}}(h+\log q(\hat{Y}_{u}))\right]^{+}}\\
{\displaystyle \in\frac{1}{\sqrt{|\mathbb{T}_{N}^{d}|}}\mathbf{E}\left[\sum_{i=0}^{2^{d}}\sum_{{\scriptscriptstyle u\in\mathcal{S}_{i}}}(h+\log q(Y_{u}))\right]^{+}\pm\frac{O(\mathbf{E}|\mathcal{J}_{N,\lfloor\delta_{N}N\rfloor}^{c}|)}{\sqrt{|\mathbb{T}_{N}^{d}|}}}\\
{\displaystyle =\frac{\text{Var}(\log q(Y_{0}))}{\sqrt{2\pi}}\pm o(1)\pm\frac{O(\mathbf{E}|\mathcal{J}_{N,\lfloor\delta_{N}N\rfloor}^{c}|)}{\sqrt{|\mathbb{T}_{N}^{d}|}}}
\end{array}
\]
By \ref{enu:} of the present proposition, the last summand is of
order $o(1)$. This yields the claim in \ref{enu:-2} and completes
the proof.
\end{proof}
We are now ready to prove Theorem \ref{thm:For-any-i.i.d.-1}.

\noindent\textbf{Theorem \ref{thm:For-any-i.i.d.-1}.} \textit{For
any i.i.d. $\mathbb{Z}^{d}$-processes $X$ and $Y$ of equal entropy
$h$, if there exists a finitary map $\phi:X\rightarrow Y$ with $\mathbf{E}[R_{\phi}^{d/2}]<\infty$,
then $\text{Var}(I(X_{0}))=\text{Var}(I(Y_{0}))$.}
\begin{proof}
Fix $N\in\mathbb{N}$, and consider the $(\mathbb{T}_{N}^{d},\lfloor\delta_{N}N\rfloor)$-modeling
of $\phi$, consisting of the process $\hat{X}\sim(A^{\mathbb{T}_{N}^{d}},p^{\mathbb{T}_{N}^{d}})$
and the mapping $\hat{\phi}^{\lfloor\delta_{N}N\rfloor}(\hat{X})=\hat{Y}$.
Denoting as before the marginals of $X$ and $Y$ by $p$ and $q$
respectively, Lemma \ref{lem:For-any-,} asserts that a.s.
\[
\prod_{u\in\mathbb{T}_{N}^{d}}p(\hat{X}_{u})\leq\prod_{\begin{array}{c}
u\in\mathbb{T}_{N}^{d}\\
\hat{Y}_{u}\in B
\end{array}}q(\hat{Y}_{u}).
\]
Taking logarithms on both sides, adding $|\mathbb{T}_{N}^{d}|h$ to
both sides and then subtracting one side from the other and factorizing
by $\sqrt{|\mathbb{T}_{N}^{d}|}$, gives
\begin{equation}
\underbrace{\frac{1}{\sqrt{|\mathbb{T}_{N}^{d}|}}\sum_{u\in\mathbb{T}_{N}^{d}}\left(h+\log p(\hat{X}_{u})\right)}_{A}-\underbrace{\frac{1}{\sqrt{|\mathbb{T}_{N}^{d}|}}\sum_{u\in\mathbb{T}_{N}^{d}}\left(h+\mathbf{1}_{\{\hat{Y}_{u}\in B\}}\log q(\hat{Y}_{u})\right)}_{B}\leq0\label{eq:-3}
\end{equation}

Denote by $A$ and $B$ the two summands on the left hand side of
the above inequality. \foreignlanguage{american}{Notice that for any
fixed $u\in\mathbb{T}_{N}^{d}$, using the coupling of Proposition
}\ref{prop:}\ref{enu:Suppose--is}\foreignlanguage{american}{ with
$\mathcal{S}=\{u\}$, one has
\begin{align*}
\mathbf{E}\left[\mathbf{1}_{\{\hat{Y}_{u}\in B\}}\log q(\hat{Y}_{u})\right] & =\mathbf{E}\left[\mathbf{1}_{\{\hat{Y}_{u}\in B\}}\log q(\hat{Y}_{u})-\log q(Y_{u})\right]+\mathbf{E}\left[\log q(Y_{u})\right]\\
 & =O(\mathbf{E}\mathbf{1}_{\{\hat{Y}_{u}=*\}})-h.
\end{align*}
}

\selectlanguage{american}%
This implies that\foreignlanguage{english}{ the expectation of the
left hand side of (\ref{eq:-3}) equals to 
\[
\mathbf{E}[A-B]=O\left(\frac{\mathbf{E}|\{u:\hat{Y}_{u}=*\}|}{\sqrt{|\mathbb{T}_{N}^{d}|}}\right)=O\left(\frac{\mathbf{E}[|\mathcal{J}_{N,\lfloor\delta_{N}N\rfloor}^{c}|]}{\sqrt{|\mathbb{T}_{N}^{d}|}}\right)
\]
which by Proposition \ref{prop:The-following-holds:}\ref{enu:} goes
to zero as $N\rightarrow\infty$. For any two reals $a,b\in\mathbb{R}$
one has that 
\[
a-b=(a-b)^{+}-(b-a)^{+}
\]
 and since $(A-B)^{+}=0$ (by (\ref{eq:-3})), we get 
\[
\mathbf{E}(B-A)^{+}=\mathbf{E}(A-B)^{+}+o(1)=o(1).
\]
}

\selectlanguage{english}%
Using the triangle inequality $a^{+}-b^{+}\leq(a-b)^{+}$, we get
\begin{align*}
\mathbf{E}A^{+}-\mathbf{E}B^{+} & \leq\mathbf{E}(A-B)^{+}=0\\
\mathbf{E}B^{+}-\mathbf{E}A^{+} & \leq\mathbf{E}(B-A)^{+}=o(1)
\end{align*}
which yields
\[
\mathbf{E}B^{+}=\mathbf{E}A^{+}+o(1)
\]
Using the identities $\mathbf{E}A^{+}=\frac{\text{Var}(\log p(X_{0}))}{\sqrt{2\pi}}\pm o(1)$
and $\mathbf{E}B^{+}=\frac{\text{Var}(\log q(Y_{0}))}{\sqrt{2\pi}}\pm o(1)$
of Proposition \ref{prop:The-following-holds:} and taking $N\nearrow\infty$,
we get
\[
\frac{\text{Var}(\log p(X_{0}))}{\sqrt{2\pi}}=\frac{\text{Var}(\log q(Y_{0}))}{\sqrt{2\pi}}
\]
and the proof is complete.
\end{proof}

\section{\protect\label{sec:A-Schmidt-type-theorem}A Schmidt-type theorem}

In this section we will prove Theorem \ref{thm:Let--and-1}. Some
preparatory background will be presented, followed by the proof itself.
\begin{prop}
\label{lem:Suppose-that-for}Suppose that for $i=1,2$ and all $N\in\mathbb{N}$,
$f_{i}=f_{i,N}$ are some random variables with
\begin{equation}
||f_{1}-f_{2}||_{1}=e^{-\Omega(N)}\label{eq:-16}
\end{equation}
and
\begin{equation}
||f_{i}||_{\infty}=N^{O(1)}\quad(i=1,2).\label{eq:-17}
\end{equation}

Then, for any $k\in\mathbb{N}$, 
\[
\left|||f_{1}||_{k}^{k}-||f_{2}||_{k}^{k}\right|=e^{-\Omega_{k}(N)}
\]
\end{prop}

The lowercase $k$ in the notation $\Omega_{k}(N)$ signifies that
the constants provided by the $\Omega(N)$ notation depend on the
value of $k$.
\begin{proof}
We have
\begin{align*}
||f_{1}-f_{2}||_{k} & =\left(||(f_{1}-f_{2})^{k-1}(f_{1}-f_{2})||_{1}\right)^{1/k}\\
 & \leq||(f_{1}-f_{2})^{k-1}||_{\infty}^{1/k}||f_{1}-f_{2}||_{1}^{1/k}\\
 & =N^{O(1)}O(e^{-N/k})\\
 & =e^{-\Omega_{k}(N)},
\end{align*}
(here we used Hölder's inequality $||fg||_{1}\leq||f||_{\infty}||g||_{1}$.)

We now have
\begin{align*}
||f_{1}||_{k}^{k} & =||f_{1}-f_{2}+f_{2}||_{k}^{k}\\
 & \leq\left(||f_{1}-f_{2}||_{k}+||f_{2}||_{k}\right)^{k}\\
 & =\left(e^{-\Omega(N)}+||f_{2}||_{k}\right)^{k}\\
 & =||f_{2}||_{k}^{k}+e^{-\Omega(N)}N^{O(k-1)}\\
 & =||f_{2}||_{k}^{k}+e^{-\Omega_{k}(N)}.
\end{align*}

This inequality, together with the one obtained by flipping the roles
of $f_{1}$ and $f_{2}$ above, implies the claimed inequality.
\end{proof}
\begin{prop}
\label{lem:For-any-,-1}For any $d,k\in\mathbb{N}$, $2k<m\in\mathbb{N}$
and $N=m\cdot2(k+1)$, given a set $\mathcal{S}\subset\mathbb{T}_{N}^{d}$
of size $|\mathcal{S}|\leq k$, there is some $\boldsymbol{v}\in\mathbb{Z}^{d}$
so that under the natural quotient map $\pi:\mathbb{Z}^{d}\rightarrow\mathbb{T}_{N}^{d}$,
\[
\bigcup_{\boldsymbol{t}\in\pi^{-1}(\mathcal{S})\cap(\boldsymbol{v}+[1,N]^{d})}\boldsymbol{t}+[-m,m]^{d}\subset\boldsymbol{v}+[1,N]^{d}.
\]
\end{prop}

\begin{proof}
For each coordinate $1\leq i\leq d$, the $i$th coordinate values
of the set $B_{\mathbb{T}_{N}^{d}}(\mathcal{S},m)$ doesn't contain
all the values in $[1,N]$, as:
\begin{align*}
|\{u_{i}:u\in B_{\mathbb{T}_{N}^{d}}(\mathcal{S},m)\}| & \leq|\mathcal{S}|\cdot(2m+1)\\
 & \leq k\cdot(2m+1)\\
 & \leq N-2m+k\\
 & <N.
\end{align*}
Writing $\sigma:\mathbb{Z}\rightarrow\mathbb{T}$ for the one-dimensional
natural quotient map, pick any $v_{i}\in[1,N]\backslash\sigma^{-1}\left(\{u_{i}:\boldsymbol{u}\in B_{\mathbb{T}_{N}^{d}}(\mathcal{S},m)\}\right)$.
Its choice implies that both $v_{i},v_{i}+N$ are not in the set $\sigma^{-1}\left(\{u_{i}:\boldsymbol{u}\in B_{\mathbb{T}_{N}^{d}}(\mathcal{S},m)\}\right)$.
As this set is a union of segments in $\mathbb{Z}$, it follows that
each of the segments in $\sigma^{-1}\left(\{u_{i}:u\in B_{\mathbb{T}_{N}^{d}}(\mathcal{S},m)\}\right)$
is either completely included in $v_{i}+[1,N]$ or completely excluded
from it. Since for each $\boldsymbol{s}=(s_{1},...,s_{d})\in\mathcal{S}$
there exists a unique candidate $\{t\}=\sigma^{-1}(s_{i})\cap(v_{i}+[1,N])$,
its segment $[t-m,t+m]$ must be completely included in $v_{i}+[1,N]$.
Thus we get:
\begin{equation}
(v_{i}+[1,N])\supset\bigcup_{t\in\sigma^{-1}(\{s_{i}:\boldsymbol{s}\in S\})\cap(v_{i}+[1,N])}[t-m,t+m].\label{eq:-10}
\end{equation}

Letting $\boldsymbol{v}=(v_{1},...,v_{d})$ and observing (\ref{eq:-10})
for each of the coordinates $1\leq i\leq d$ gives
\begin{align*}
\bigcup_{t\in\pi^{-1}(\mathcal{S})\cap(\boldsymbol{v}+[1,N]^{d})}t+[-m,m]^{d} & \subset\mathsf{X}_{i=1}^{d}\left(\bigcup_{t\in\sigma^{-1}(\{s_{i}:\boldsymbol{s}\in S\})\cap(v_{i}+[1,N])}[t-m,t+m]\right)\\
 & \subset\mathsf{X}_{i=1}^{d}\left(v_{i}+[1,N]\right)\\
 & =\boldsymbol{v}+[1,N]^{d}
\end{align*}
as claimed.
\end{proof}
We are now ready to prove Theorem \ref{thm:Let--and-1}:

\noindent\textbf{Theorem \ref{thm:Let--and-1}.} \textit{Let $X$
and $Y$ be two i.i.d. $\mathbb{Z}^{d}$-processes of equal entropy
$h$, with marginal measures $p,q$ respectively. Suppose that for
some constant $c>0$ there is a homomorphism $\phi(X)=Y$ that satisfies
$\mathbf{P}[R_{\phi}>n]=O(e^{-cn})$. Then $p=q$ up to permutation
on the symbols.}
\begin{proof}
We will show by induction that the information functions applied to
the zero marginals, $I_{p}(X_{0})=\log p(X_{0})^{-1}$ and $I_{q}(Y_{0})=\log q(Y_{0})^{-1}$,
agree on all their moments:
\[
\begin{array}{cr}
\mathbf{E}[I_{p}(X_{0})^{k}]=\mathbf{E}[I_{q}(Y_{0})^{k}] & \qquad\forall k\in\mathbb{N}\end{array},
\]
 and so, being bounded random variables (in fact, finite valued),
they share the same distributions. By assumption, both agree on their
first moment, being the processes entropy. Let $k>1$ and suppose
that $\mathbf{E}[I_{p}(X_{0})^{i}]=\mathbf{E}[I_{q}(Y_{0})^{i}]$
for all $1\leq i<k$. 

Notice that we can express the $k$th moment of $\log p(X_{0})^{-1}$
by
\begin{align}
\mathbf{E}[I_{p}(X_{0})^{k}] & =\frac{1}{N^{d}}\mathbf{E}[\sum_{u\in[1,N]^{d}}I_{p}(X_{u})^{k}]\label{eq:-8}\\
 & =\frac{1}{N^{d}}\mathbf{E}[\sum_{u\in[1,N]^{d}}I_{p}(X_{u})]^{k}-\mathbf{E}[F(\{I_{p}(X_{u})\}_{u\in[1,N]^{d}})]\\
 & =\frac{1}{N^{d}}\mathbf{E}[\sum_{u\in[1,N]^{d}}I_{p}(X_{u})]^{k}-f(\{\mathbf{E}[I_{p}(X_{u})^{i}]\}_{1\leq i<k})]\label{eq:-5}
\end{align}
where $N>0$, $F=F\left(\left\{ x_{u}\right\} _{u\in\mathbb{T}_{N}^{d}}\right)$
is some multinomial of degree $k$ on $N^{d}$ variables, none of
whose degree-$k$ monomials consists of one variable, and $f$ is
some muiltinomial on $k-1$ variables (the equality (\ref{eq:-5})
follows from the i.i.d. property of $\{I_{p}(X_{u})\}_{u\in[1,N]^{d}}$).
Similarly we have
\[
\mathbf{E}[I_{q}(Y_{0})^{k}]=\frac{1}{N^{d}}\mathbf{E}[\sum_{u\in[1,N]^{d}}I_{q}(Y_{u})]^{k}-f(\{\mathbf{E}[I_{q}(Y_{u})^{i}]\}_{1\leq i<k})].
\]

By the induction assumption, 
\[
\begin{array}{cr}
\mathbf{E}[I_{p}(X_{0})^{i}]=\mathbf{E}[I_{q}(Y_{0})^{i}] & \qquad1\leq i<k\end{array},
\]
which implies 
\[
f(\{\mathbf{E}[I_{p}(X_{u})^{i}]\}_{1\leq i<k})=f(\{\mathbf{E}[I_{q}(Y_{u})^{i}]\}_{1\leq i<k})].
\]

Thus proving the equality $\mathbf{E}[I_{p}(X_{0})^{k}]=\mathbf{E}[I_{q}(Y_{0})^{k}]$
is equivalent to showing that
\begin{equation}
\frac{1}{N^{d}}\mathbf{E}[\sum_{u\in[1,N]^{d}}I_{p}(X_{u})]^{k}=\frac{1}{N^{d}}\mathbf{E}[\sum_{u\in[1,N]^{d}}I_{q}(Y_{u})]^{k}.\label{eq:-9}
\end{equation}

Let us restrict our discussion to all those $N$ which are a multiple
$N=m\cdot2(k+1)$, with $m$ being any integer $m>k$. For such $N$
(and its respective $m=N/(2(k+1))$), consider the $(\mathbb{T}_{N}^{d},m)$-modeling
of $\phi$ as defined in Definition \ref{def:Letting--and}. Clearly,
in the left hand side of (\ref{eq:-9}), the sites of $X$ can be
replaced with sites of $\hat{X}=\left(\hat{X}_{u\in\mathbb{T}_{N}^{d}}\right)$:
\begin{equation}
\frac{1}{N^{d}}\mathbf{E}[\sum_{u\in[1,N]^{d}}I_{p}(X_{u})]^{k}=\frac{1}{|\mathbb{T}_{N}^{d}|}\mathbf{E}[\sum_{u\in\mathbb{T}_{N}^{d}}I_{p}(\hat{X}_{u})]^{k}.\label{eq:-14}
\end{equation}

We show now that one can replace the sites of $Y$ in the right hand
side of (\ref{eq:-9}) with sites of $\hat{Y}=\left(\hat{Y}_{u\in\mathbb{T}_{N}^{d}}\right)$
with negligible change in the expression's value. More precisely,
we claim that
\begin{equation}
\frac{1}{N^{d}}\mathbf{E}[\sum_{u\in[1,N]^{d}}I_{q}(Y_{u})]^{k}=\frac{1}{|\mathbb{T}_{N}^{d}|}\mathbf{E}[\sum_{u\in\mathbb{T}_{N}^{d}}\mathbf{1}_{\{\hat{Y}_{u}\in B\}}I_{q}(\hat{Y}_{u})]^{k}\pm e^{-\Omega(N)}.\label{eq:-12}
\end{equation}

To see why this holds, note that the expression $\mathbf{E}[\sum_{u\in[1,N]^{d}}I_{q}(Y_{u})]^{k}$
can be written as a sum of expectations of products, of the kind
\begin{equation}
\mathbf{E}[\sum_{u\in[1,N]^{d}}I_{q}(Y_{u})]^{k}=\sum_{\mathcal{S}\in([1,N]^{d})^{\times k}}\mathbf{E}[{\displaystyle \prod_{u\in\mathcal{S}}}I_{q}(Y_{u})].\label{eq:-13}
\end{equation}

Fix such an $\mathcal{S}$ and consider its projection $\pi(\mathcal{S})\subset\mathbb{T}_{N}^{d}$
under the natural quotient map $\pi:\mathbb{Z}^{d}\rightarrow\mathbb{T}_{N}^{d}$.
With our assumption that $m=N/2(k+1)>k$, Proposition \ref{lem:For-any-,-1}
provides some $v\in\mathbb{Z}^{d}$ so that
\[
\bigcup_{t\in\pi^{-1}(\pi(\mathcal{S}))\cap(v+[1,N]^{d})}t+[-m,m]^{d}\subset v+[1,N]^{d},
\]
so that one can use the coupling of Proposition \ref{prop:}.\ref{enu:Suppose--is}
to get that 
\[
\left|\mathbf{E}[\prod_{u\in\pi(\mathcal{S})}\mathbf{1}_{\{\hat{Y}_{u}\in B\}}I_{q}(\hat{Y}_{u})]-\mathbf{E}[{\displaystyle \prod_{u\in\pi^{-1}(\pi(\mathcal{S}))\cap(v+[1,N]^{d})}}I_{q}(Y_{u})]\right|=O(ke^{-N}).
\]

Applying the above coupling for all $N^{d\cdot k}$ summands in (\ref{eq:-13})
enlarges the error term to $O(N^{d\cdot k}ke^{-N})=e^{-\Omega(N)}$,
and equation (\ref{eq:-12}) follows as we claimed.

By (\ref{eq:-14}) and (\ref{eq:-12}) (which we have just proved),
the proof of (\ref{eq:-9}), and as a result the theorem itself, amounts
to showing that
\begin{equation}
\left|\frac{1}{|\mathbb{T}_{N}^{d}|}\mathbf{E}[\sum_{u\in\mathbb{T}_{N}^{d}}\mathbf{1}_{\{\hat{Y}_{u}\in B\}}I_{q}(\hat{Y}_{u})]^{k}-\frac{1}{|\mathbb{T}_{N}^{d}|}\mathbf{E}[\sum_{u\in\mathbb{T}_{N}^{d}}I_{p}(\hat{X}_{u})]^{k}\right|=o(1),\label{eq:-15}
\end{equation}
and the rest of the proof is dedicated to that.

We have for a.e. $\hat{X}$,
\begin{equation}
N^{d}\cdot\max_{a\in A}I_{p}(a)\geq\sum_{u\in\mathbb{T}_{N}^{d}}I_{p}(\hat{X}_{u})-\sum_{u\in\mathbb{T}_{N}^{d}}\mathbf{1}_{\{\hat{Y}_{u}\in B\}}I_{q}(\hat{Y}_{u})\geq0\label{eq:-4}
\end{equation}
where the left inequality is just a crude bound, and the right inequality
is implied by Lemma \ref{lem:For-any-,}. Furthermore, using the coupling
of Proposition \ref{prop:}.\ref{enu:Suppose--is}, one has
\begin{align*}
\mathbf{E}[\mathbf{1}_{\{\hat{Y}_{u}\in B\}}I_{q}(\hat{Y}_{u})] & =h\pm O(\mathbf{P}[R_{\phi}>m])\\
 & =h\pm e^{-\Omega(m)}\\
 & =h\pm e^{-\Omega(N)}
\end{align*}
so that the expectation of the absolute value of the middle expression
in (\ref{eq:-4}) is bounded by 
\begin{align}
\begin{array}{l}
\mathbf{E}\left|\sum_{u\in\mathbb{T}_{N}^{d}}I_{p}(\hat{X}_{u})-\sum_{u\in\mathbb{T}_{N}^{d}}\mathbf{1}_{\{\hat{Y}_{u}\in B\}}I_{q}(\hat{Y}_{u})\right|=\\
\begin{array}{lrl}
 &  & =\mathbf{E}[\sum_{u\in\mathbb{T}_{N}^{d}}I_{p}(\hat{X}_{u})-\sum_{u\in\mathbb{T}_{N}^{d}}\mathbf{1}_{\{\hat{Y}_{u}\in B\}}I_{q}(\hat{Y}_{u})]\\
 &  & =\sum_{u\in\mathbb{T}_{N}^{d}}\mathbf{E}[I_{p}(\hat{X}_{u})]-\sum_{u\in\mathbb{T}_{N}^{d}}\mathbf{E}[\mathbf{1}_{\{\hat{Y}_{u}\in B\}}I_{q}(\hat{Y}_{u})]\\
 &  & =N^{d}h-N^{d}h\pm N^{d}e^{-\Omega(N)}\\
 &  & =e^{-\Omega(N)}.
\end{array}
\end{array}\label{eq:-11}
\end{align}

Taking 
\begin{align*}
f_{1,N} & =\sum_{u\in\mathbb{T}_{N}^{d}}I_{p}(\hat{X}_{u})\\
f_{2,N} & =\sum_{u\in\mathbb{T}_{N}^{d}}\mathbf{1}_{\{\hat{Y}_{u}\in B\}}I_{q}(\hat{Y}_{u})
\end{align*}
we have that $||f_{i,N}||_{\infty}=O(N^{d})$ for both $i=1,2$, and
that $||f_{1,N}-f_{2,N}||_{1}=e^{-\Omega(N)}$ by (\ref{eq:-11}).
Thus we can use Proposition \ref{lem:Suppose-that-for} to conclude
that (\ref{eq:-15}) takes place, as we claimed. This proves the induction
step and completes the proof of the theorem.
\end{proof}

\lyxaddress{uriel.gabor@gmail.com}

\lyxaddress{Einstein Institute of Mathematics}

\lyxaddress{The Hebrew University of Jerusalem}

\lyxaddress{Edmond J. Safra Campus, Jerusalem, 91904, Israel}
\end{document}